\documentclass{article}
\usepackage[english]{babel} 
\usepackage[utf8]{inputenc}
\usepackage{mathtools}
\usepackage{amssymb}
\usepackage{amsthm}

\usepackage{dirtytalk}
\usepackage{xcolor}
\usepackage{framed}
\definecolor{shadecolor}{RGB}{153,204,255}

\newtheorem{theorem}{Theorem}[section]
\newtheorem{definition}[theorem]{Definition}

\newtheorem{proposition}[theorem]{Proposition}

\theoremstyle{remark}

\newcommand{\defhigh}[1]{\textsc{#1}}

\newcommand{\redsolovay}[1][\le]{\ensuremath{{#1}_{\mathrm{S}} } }

\newcommand{\redsolovayzweia}[1][\le]{\ensuremath{{#1}_{\mathrm{S}}^{\mathrm{2a}} } }

\title{Solovay reducibility implies S2a-reducibility}

\author{Ivan Titov}

\begin{document}

\maketitle

\begin{abstract}
The original notion of Solovay reducibility was introduced by Robert M.\ Solovay~\cite{Solovay-1975} in 1975 as a measure of relative randomness. 

The {S2a}-reducibility introduced by Xizhong Zheng and Robert Rettinger~\cite{Zheng-Rettinger-2004} in 2004 is a modification of Solovay reducibility suitable for computably approximable (c.a.) reals.

We demonstrate that Solovay reducibility implies {S2a}-reducibility on the set of c.a.\ reals, even with the same constant, but not vice versa.
\end{abstract}

\section{Introduction and background}
We assume the reader to be familiar with the basic concepts and results of algorithmic randomness. Our notation is standard. Unexplained notation can be found in Downey and Hirschfeldt  \cite{Downey-Hirschfeldt-2010}. As it is standard in the field, all rational and real numbers are meant to be in the unit interval $[0,1)$, unless stated otherwise.

We start by reviewing some central concepts and results that will be used subsequently.

\begin{definition}
\begin{enumerate}
    \item
    A \defhigh{computable approximation} is a computable Cauchy sequence, i.e., a computable sequence of rational numbers that converges.
    A real is \defhigh{computably approximable}, or \defhigh{c.a.}, if it is the limit of some computable approximation.
    \item
    A \defhigh{left-c.e.\ approximation} is a nondecreasing computable approximation. 
    A real is \defhigh{left-c.e.}\ if it is the limit of some left-c.e.\ approximation.
    \item
    A \defhigh{right-c.e.\ approximation} is a nonincreasing computable approximation. 
    A real is \defhigh{right-c.e.}\ if it is the limit of some right-c.e.\ approximation.
\end{enumerate}
\end{definition}
In particular, if $\alpha$ is a left-c.e.\ real with a left-c.e.\ approximation~$a_0,a_1,\dots$, then~$1-\alpha$ is right-c.e.\ real with a right-c.e.\ approximation~$1-a_0,1-a_1,\dots$ .
\begin{definition}
    The \defhigh{left cut} of a real $\alpha$, written~$LC(\alpha)$, is the set of all rationals strictly smaller than $\alpha$.
\end{definition}

\begin{definition}[Solovay~\cite{Solovay-1975}, 1975]\label{Solovay-reducibility}
    A real~$\alpha$ is \defhigh{Solovay reducible} to a real~$\beta$, written~$\alpha\redsolovay\beta$, if there exist a constant~$c\in\mathbb{R}$ and a partially computable function $g$ from the set~$\mathbb{Q}\cap[0,1)$ to itself such that, for all~$q< \beta$, the value $g(q)$ is defined and fulfills the inequality
    \begin{equation}\label{eq:Solovay-reducibility}
        0<\alpha - g(q)<c(\beta - q).
    \end{equation}
\end{definition}
We will refer to~\eqref{eq:Solovay-reducibility} as \defhigh{Solovay condition} and to~$c$ as \defhigh{Solovay constant}. 

We say that the function $g$ \defhigh{witnesses} the Solovay reducibility of $\alpha$ to $\beta$ in case~$g$ satisfies the conditions in Definition~\ref{Solovay-reducibility} for some Solovay constant. Note that such a function $g$ is defined on the whole set~$LC(\beta)$, maps it to $LC(\alpha)$, and fulfills
\begin{equation}\label{eq:translation-function}
    \lim\limits_{q\nearrow\beta} g(q) = \alpha,
\end{equation}
where $\lim\limits_{q\nearrow\beta}$ denotes the left limit.

The set of left-c.e.\ reals is downwards closed under Solovay reducibility. This is usually shown using the index characterization of Solovay reducibility on the set of left-c.e.\ reals by Calude, Hertling, Khoussainov, and Wang~\cite{Calude-etal-1998}, whereas the following proof uses the original definition.
\begin{proposition}\label{closure-of-left-ce}
    If a real $\beta$ is left-c.e., and a real $\alpha$ is Solovay reducible to $\beta$, then $\alpha$ is left-c.e. 
\end{proposition}

\begin{proof}
Let $\beta$ be a left-c.e.\ real, and let $\alpha\redsolovay\beta$ with Solovay constant~$c$ witnessed by some partial function~$g$. Fix some left-c.e.\ approximation $b_0,b_1,\ldots$ of $\beta$, and define a computable sequence of rationals $a_0,a_1,\ldots$ by setting $a_n = g(b_n)$ for all~$n$. By assumption on~$g$, we obtain for every~$n$ the inequalities~$a_n < \alpha$ and
\begin{equation}
    0 < \alpha - a_n < c(\beta - b_n).
\end{equation}
By~$\lim\limits_{n\to\infty}b_n = \beta$, these inequalities imply that~$\lim\limits_{n\to\infty} a_n = \alpha$, hence~$a_0,a_1,\dots$ is a computable approximation of~$\alpha$.
Finally, setting $a'_n:=\max\{a_m : m\leq n\}$ for every $n$ provides a left-c.e\ approximation~$a'_0,a'_1,\ldots$ of~$\alpha$, hence $\alpha$ is left-c.e. 
\end{proof}
\medskip

Since the Solovay condition~\eqref{eq:Solovay-reducibility} is required only for rationals $q$ in the left cut of $\beta$, many researchers focused on Solovay reducibility as a measure of relative randomness of left-c.e.\ reals, while, outside of the left-c.e.\ reals, the notion has been considered  as "badly behaved" by several authors (see e.g.\ Downey and Hirschfeldt~\cite[Section~9.1]{Downey-Hirschfeldt-2010}).

Zheng and Rettinger~\cite{Zheng-Rettinger-2004} introduced  the reducibility $\redsolovayzweia$ on the set of c.a.\ reals, which is equivalent to $\redsolovay$ on the set of left-c.e.\ reals and, nowadays, is just called Solovay reducibility by some authors, even though it differs from Solovay reducibility as introduced in Definition~\ref{Solovay-reducibility} on the set of c.a.\ reals.

\begin{definition}[Zheng, Rettinger~\cite{Zheng-Rettinger-2004}, 2004]
\label{S2a-reducibility}
    A c.a.\ real $\alpha$ is {S2a}-\defhigh{reducible} to a c.a.\ real $\beta$, written $\alpha \redsolovayzweia \beta$, if there exist a constant $c\in\mathbb{R}$ and computable approximations $a_0,a_1,\ldots$ and~$b_0,b_1,\ldots$ of~$\alpha$ and~$\beta$, respectively, that fulfill for every $n\in\mathbb{N}$ the inequality
    \begin{equation}\label{eq: 2a-property}
        |\alpha - a_n| \leq c (|\beta - b_n| + 2^{-n}).
    \end{equation}
\end{definition}
We  will refer to~\eqref{eq: 2a-property} as~{S2a}-\defhigh{condition} and to~$c$ as~{S2a}-\defhigh{constant} and say that the sequences~$a_0,a_1,\dots$ and~$b_0,b_1,\dots$ \defhigh{witness} the {S2a}-reducibility of~$\alpha$ to~$\beta$ in case they satisfy the conditions of Definition~\ref{S2a-reducibility} for some {S2a}-constant.

In the next section, we will show that on the set of c.a.\ reals, Solovay reducibility implies {S2a}-reducibility but not vice versa.
 
\section{The theorem}\label{m-implies-2a}

\begin{theorem}\label{the-theorem}
    Let~$\alpha$ and~$\beta$ be reals where~$\beta$ is computably approximable and $\alpha$ is Solovay reducible to~$\beta$ with Solovay constant $c$. Then $\alpha$ is computably approximable as well, and $\alpha$ is {S2a}-reducible to $\beta$ with {S2a}-constant $c$.
%
\end{theorem}

\begin{proof}
Let the partially computable function~$g$ witness that $\alpha$ is Solovay reducible to $\beta$ with Solovay constant $c$. 

%
In case $\beta$ is left-c.e., we directly obtain by Proposition~\ref{closure-of-left-ce} that $\alpha$ is left-c.e., as well, hence is computably approximable, and that thus we have~$\alpha\redsolovayzweia\beta$ with constant $c$ since $\redsolovay$ and $\redsolovayzweia$ coincide with preserving the constant $c$ on the set of left-c.e.\ reals by Zheng and Rettinger~\cite[Theorem~3.2(2)]{Zheng-Rettinger-2004}. So, in what follows, assume that $\beta$ is not left-c.e. In particular, $\beta\neq 0$.

Let $b_0,b_1,\ldots$ and~$q_0,q_1,\ldots$ be a computable approximation of $\beta$ and a computable enumeration of $\mathrm{dom}(g)$, respectively, where, without loss of generality, we may assume that~$b_0 = q_0 = 0$ since, by~$\beta>0$, $g(0)$ is defined.

Then, it suffices to inductively construct an infinite strictly increasing computable index sequence $i_0,i_1,\ldots$ and a computable sequence $a_0,a_1,\ldots$ of rationals that fulfill the inequality
\begin{equation}\label{eq:2a-to-prove}
    |\alpha - a_n| < c(|\beta - b_{i_n}| + 2^{-n})
\end{equation}
for every $n$. Note that the sequence~$b_0,b_1,\dots$ and its subsequence~$b_{i_0},b_{i_1},\dots$ have the same limit~$\beta$, hence by~\eqref{eq:2a-to-prove} the sequence $a_0,a_1,\dots$ converges to $\alpha$.

\paragraph{Construction of the sequences}

At step $0$, we set 
\[i_0 = 0 \makebox[5em]{and}a_0 = g(q_{i_0}) = g(q_0) = g(0).\]

\medskip
At step $n>0$, we assume that 
$i_{n-1}$ has already been defined. 
We say that a rational~$b$ \defhigh{satisfies requirement} $R_n$, if there exist a natural number~$\ell\geq 2$ and indices~$m_0,\dots,m_{\ell}$ such that the values~$g(q_{m_{0}}), \ldots, g(q_{m_{\ell}})$ are all defined and it holds that
\begin{align}
        \label{eq:m-l-sandwiched}
        b - 2^{-n-1} < q_{m_{\ell}} < b,  \quad & \\
       \label{eq:sorted-ms}
        0 = q_{m_0} < \ldots < q_{m_{\ell}},  \quad & \\
        \label{eq:distance}
        q_{m_{k+1}} - q_{m_{k}} < 2^{-n-1} \quad &
        \text{for all } k \in \{0, \ldots, \ell-1\},\\
        \label{eq:m-implies-2a-proof-g-distances}
        0 < g(q_{m_{\ell}}) - g(q_{m_k}) < c ( q_{m_{\ell}} - q_{m_{k}}+ 2^{-n-2})
        \quad &
        \text{for all } k \in \{0, \ldots, \ell-1\}.
\end{align}
At step~$n$ we search in parallel for~$i>i_{n-1}$ such that~$b=b_i$ satisfies requirement~$R_n$, and for such~$i$ that is found first, we let~$i_n = i$ and $a_n = g(q_{m_{\ell}})$ for the corresponding values of~$\ell$ and~$m_{\ell}$.

\paragraph{Every construction step terminates}

Fix $n>0$. We demonstrate that there is some $i>i_{n-1}$ such that~$b_i$ satisfies requirement~$R_n$, hence step~$n$ will be completed successfully.

By definition of {S2a}-reducibility, the partial function~$g$ is defined on all rationals~$q< \beta$, thus its domain is a dense subset of the real interval~$[0,\beta]$, which has nonzero length because of~$\beta\neq 0$.
Therefore, we can fix a natural number~$\ell\ge 2$ and indices $m_0,\ldots,m_{\ell-1}$ such that the values~$g(q_{m_{0}}), \ldots, g(q_{m_{\ell-1}})$ are all defined,~\eqref{eq:sorted-ms} and~\eqref{eq:distance} hold with the conditions on~$m_{\ell}$ removed, and we have
\begin{equation}\label{eq:m_l-with-beta}
\beta - 2^{-n-2} < q_{m_{\ell-1}} < \beta.
\end{equation}
Next, let
\begin{equation}\label{eq:A<alpha}
A := \max\limits_{k\in\{0,\ldots,\ell-1\}} g(q_{m_k}).
\end{equation}
By choice of~$g$ and~\eqref{eq:translation-function}, we have~$A < \alpha$ and $\lim\limits_{q\nearrow\beta}g(q) = \alpha$. Consequently, first, there exists a real $\varepsilon>0$ such that
\[g(q) \in (A,\alpha) \quad \text{ for all } q\in (\beta - \varepsilon, \beta).\]
Second, we can fix an index $m_{\ell} > \max\{m_0,\ldots,m_{\ell-1}\}$ such that
\begin{equation}\label{eq:three-policemen}
\max\{q_{m_{\ell-1}}, \beta -\varepsilon
\}
< q_{m_{\ell}} <\beta.
\end{equation}
Then~\eqref{eq:sorted-ms} and~\eqref{eq:distance} hold, and the inequalities
~\eqref{eq:m_l-with-beta} and~\eqref{eq:three-policemen} imply
\begin{equation}\label{eq:two-policemen}
    \beta - 2^{-n-2} < \beta - \varepsilon < q_{m_{\ell}} < \beta.
\end{equation}

In order to show~\eqref{eq:m-implies-2a-proof-g-distances}, fix~$k \in\{0, \ldots, \ell-1\}$. The first inequality in~\eqref{eq:m-implies-2a-proof-g-distances} holds because, by choice of~$\varepsilon$ and of~$m_{\ell}$, we have~$g(q_{m_k}) < g(q_{m_{\ell}})$. The second one holds because we have
\begin{align*}
g(q_{m_{\ell}}) - g(q_{m_k}) 
&< \alpha - g(q_{m_k})    
< c(\beta - q_{m_k})\\
&= c(\beta - q_{m_{\ell}}) + c(q_{m_{\ell}} - q_{m_k}) 
< c(2^{-n-2} + q_{m_{\ell}} - q_{m_k}),
\end{align*}
where the first inequality holds by the left part of~\eqref{eq:Solovay-reducibility} for~$q_{m_{\ell}}<\beta$,
the second one holds by the right part of~\eqref{eq:Solovay-reducibility} for~$q_{m_k}<\beta$,
and the third one by~\eqref{eq:two-policemen}.


Furthermore,~\eqref{eq:two-policemen} implies that all~$b$ that are close enough to~$\beta$ fulfill~\eqref{eq:m-l-sandwiched}. Now the~$b_i$ converge to~$\beta$, hence, for almost all~$i$, the value~$b = b_i$ fulfills~\eqref{eq:m-l-sandwiched}. In particular, step~$n$ will complete successfully after finding~$i>i_{n-1}$ as required.
   
\paragraph{The constructed sequences witness the {S2a}-reducibility}
In order to obtain $\alpha\redsolovayzweia\beta$ with the {S2a}-constant $c$, it suffices to prove for every $n$ the {S2a}-condition
\begin{equation}\label{S2a-property-to-prove-neu}
    |\alpha - a_n| < c(|\beta - b_{i_n}|+2^{-n}).
\end{equation}
So fix~$n$. Let~$\ell$ and~$m_0,\ldots,m_{\ell}$ be the witnesses found at step~$n$ for the fact that~$b_{i_n}$ satisfies requirement $R_n$. Thus, in particular, we have 
\begin{equation}\label{eq:define-a_n}
    a_n = g(q_{m_{\ell}}).
\end{equation}
We prove~\eqref{S2a-property-to-prove-neu} by distinguishing cases according to the position of~$\beta$ with respect to the values~$0 = q_{m_0} < \cdots < q_{m_{\ell}}$.
\begin{itemize}
    \item 
    In case~$q_{m_{\ell}}<\beta$, we have
    \[0 < \alpha - g(q_{m_{\ell}}) < c(\beta - q_{m_{\ell}}) \le c(|\beta - b_{i_n}| + |b_{i_n} - q_{m_{\ell}}|) < c(|\beta - b_{i_n}| + 2^{-n-1}),\]
    where the first two inequalities follow from the Solovay condition, and the last one holds by~\eqref{eq:m-l-sandwiched}.


    \item 
    In case there exists~$k\in\{0,\ldots,\ell-1\}$ such that $q_{m_k} < \beta \leq q_{m_{k+1}}$, on one hand, we have
    \begin{equation*}\label{eq:bound-1}
        \alpha - g(q_{m_{\ell}}) \leq \alpha - g(q_{m_k}) < c(\beta - q_{m_k}) < c\cdot 2^{-n-1} < (|\beta - b_{i_n}| + 2^{-n}),
    \end{equation*}
    where the first three inequalities follow, from left to right, by the lower bound in~\eqref{eq:m-implies-2a-proof-g-distances}, by the Solovay condition, and, finally, because~$\beta$ is contained in the interval~$(q_{m_k},q_{m_{k+1}}]$, which has length strictly less than~$2^{-n-1}$.

    On the other hand, we also obtain an upper bound for~$g(q_{m_{\ell}}) - \alpha$:
    \begin{align*}
        g(q_{m_{\ell}}) - \alpha  
        &< g\big(q_{m_{\ell}}) - g(q_{m_k}\big)
        \\
        & \leq c\big(q_{m_{\ell}} - q_{m_k} + 2^{-n-1}\big)
        \\
        &\le c\big( q_{m_{\ell}} - \beta + 2^{-n-1} + 2^{-n-1}\big)
        \\
        &< c\big( b_{i_n} - \beta + 2^{-n}\big) \le c( |b_{i_n} - \beta| + 2^{-n}),
    \end{align*}
    where the first four inequalities follow, from top to bottom,
    (i) by~$q_{m_k} < \beta$ and the choice of~$g$,
    (ii) by~\eqref{eq:m-implies-2a-proof-g-distances},
    (iii) because, by case assumption,~$\beta$ differs from~$q_{m_{k}}$ by at most~$2^{-n-1}$,
    and (iv) by~$q_{m_{\ell}} < b_{i_n}$.

    The upper bounds for~$\alpha - g(q_{m_{\ell}})$ and~$g(q_{m_{\ell}}) - \alpha$ imply together that
    \[|\alpha - g(q_{m_{\ell}})| < c(|\beta - b_{i_n}| + 2^{-n}).\]
\end{itemize}
The cases above are exhaustive with respect to the possible positions of~$\beta$, hence, by~$a_n = g(q_{m_{\ell}})$, inequality~\eqref{S2a-property-to-prove-neu} follows. This completes the proof of Theorem~\ref{the-theorem}.
\end{proof}

We conclude by observing that the converse of Theorem~\ref{the-theorem} fails, i.e., that there are c.a.\ reals $\alpha$ and $\beta$ such that~$\beta\redsolovayzweia\alpha$ but ${\beta\nleq_{\mathrm{S}}\alpha}$. In fact, $\alpha$ can be chosen to be left-c.e.

\begin{proposition}
Let~$\alpha$ be a real that is left-c.e.\ but not right-c.e. Then it holds that \[1-\alpha\redsolovayzweia\alpha \makebox[5em]{and}1-\alpha\nleq_{\mathrm{S}}\alpha.\]
\end{proposition}

\begin{proof}
The real~$\alpha$ is not right-c.e., hence $1-\alpha$ is not left-c.e.
This implies that ${1-\alpha\nleq_\mathrm{S}\alpha}$ since the set of left-c.e.\ reals is closed downwards under Solovay reducibility by Proposition~\ref{closure-of-left-ce}.

On the other hand, for a left-c.e.\ approximation $a_0,a_1,\ldots$ of $\alpha$, the sequence $1-a_0,1-a_1,\ldots$ is a right-c.e.\ approximation of $1-\alpha$.
The computable approximations $1-a_0,1-a_1,\ldots$ and $a_0,a_1,\ldots$ obviously fulfill the~{S2a}-condition~\eqref{eq: 2a-property} with constant $c=1$, hence~$1-\alpha\redsolovayzweia \alpha$.
\end{proof}

\end{document}